\documentclass{amsart}

\usepackage{hyperref}
\usepackage[capitalise]{cleveref} 
\usepackage{enumerate} 
\usepackage{mathrsfs} 
\usepackage{todonotes} 

\usepackage{chngcntr}
\counterwithin{figure}{section}

\usetikzlibrary{patterns}

\newtheorem{theorem}{Theorem}[section]
\newtheorem{corollary}[theorem]{Corollary}
\newtheorem{lemma}[theorem]{Lemma}

\theoremstyle{definition}
\newtheorem{definition}[theorem]{Definition}

\newcommand{\vb}[1]{\mathbf{#1}}

\title{Properties of regular Tangles}
\author{Rebecca M. Bowen}
\address{School of Mathematical and Statistical Sciences, Clemson University, Clemson, South Carolina 29631}
\email{rmaebowen96@gmail.com}

\author{Sadie Pruitt}
\address{Stephens County High School, Toccoa, Georgia 30577}
\email{sadie.pruitt@stephenscountyschools.org}

\author{Douglas A. Torrance}
\address{Department of Mathematical Sciences, Piedmont University, Demorest, Georgia 30535}
\email{dtorrance@piedmont.edu}

\begin{document}

\maketitle

\begin{abstract}
  A Tangle is a smooth simple closed curve formed from arcs (or ``links'') of circles with fixed radius.  Most previous study of Tangles has dealt with the case where these arcs are quarter-circles, but Tangles comprised of thirds and sixths of circles are introduced.  Together, these three families of Tangles are related to the three regular tilings of the plane by squares, regular hexagons, and equilateral triangles.  This relationship is harnessed to prove results about the number of links comprising a Tangle and the area that it encloses.
\end{abstract}

\section{Introduction}

A \textit{Tangle} is a smooth simple closed curve that is also a
spline constructed from arcs of circles with fixed radius known as \textit{links}.
In the existing literature
\cite{chan,taylor,FH,fleron,torrance2019enumeration},
these links have been quarter circles, but
we generalize the definition to include Tangles constructed using links
of arbitrary central angle.
The reader may be curious about why a capital ``T'' is used when naming these curves.  It comes from the \textit{Tangle toy} developed by Richard Zawitz
\cite{zawitz1985annular,zawitz2015tangle}.

We will focus entirely on \textit{planar} Tangles, i.e., Tangles that
lie in the Euclidean plane.  Every Tangle necessarily consists of links from circles in some packing of the plane.  Adjacent links of a Tangle either belong to the same circle or distinct kissing circles from this packing.  Such a packing in turn determines a tiling $\mathscr T$ of the plane by equilateral polygons for which all interior and exterior angles are at least $60^\circ$.  Indeed, draw a vertex at the center of every circle in the packing and connect with an edge those pairs of vertices that lie in kissing circles.  An angle less than $60^\circ$ would result in overlapping circles.  Going in reverse, every such tiling $\mathscr T$ determines a family of Tangles, which we call $\mathscr T$-Tangles.  In particular, when $\mathscr T$ is the tiling of the plane by squares, the $\mathscr T$-Tangles are exactly the well-known case of Tangles made up of quarter-circles.  From this point forward, we will refer to this family as \textit{square Tangles}.

There are two other regular tilings of the plane, known since
antiquity, by regular hexagons and equilateral triangles \cite[2.1.1]{GS}.
We will call the corresponding Tangles \textit{hexagonal Tangles} and
\textit{triangular Tangles}, respectively.
Together, we will call these three varieties of Tangles derived from
the regular tilings \textit{regular Tangles}.

The \textit{radius} of a Tangle is the radius of the circles from the
corresponding packing.
The number of links in a Tangle is its \textit{length}.
For regular Tangles, these links are quarters, thirds, and sixths of circles for square, hexagonal, and triangular Tangles, respectively.  Links are either convex or concave, and a sequence of $n$ adjacent convex links, together with one concave link on each side, is an $n$-\textit{bulb}, a term that was introduced in \cite{FH}.  An \textit{inverted $n$-bulb} is a sequence of $n$ adjacent concave links, together with one convex link on each side.

This paper is organized as follows.  In \cref{dual polyforms}, we introduce the \textit{dual polyform} of a Tangle, a combinatorial object that will be useful for proving various results.  We also prove a version of the Gauss-Bonnet theorem for Tangles.  In \cref{square tangles}, we obtain new proofs of two well-known results regarding the length and area enclosed by a square Tangle.  In Sections \ref{hexagonal tangles} and \ref{triangular tangles}, we prove analogous statements for hexagonal and triangular Tangles, respectively.

This paper is based in part on the senior capstone research projects
of the first two authors, under the supervision of the third author,
while undergraduate mathematics majors at Piedmont College (now Piedmont University) \cite{bowen,pruitt}.

\section{Dual Polyforms}\label{dual polyforms}

Note that as we trace the circles in a circle packing to create a Tangle,
we alternate between circles that lie on the interior of Tangle and circles
that lie on the exterior.
Each of these circles corresponds to a vertex in the corresponding tiling.
We assign a color to each vertex to indicate whether it corresponds to
an interior or exterior circle, say black for interior and white for
exterior.
Vertices in the square and hexagonal tilings of the plane are 2-colorable, so the boundary coloring will agree with the usual notion of vertex coloring, where adjacent vertices have distinct colors.  However, vertices in the triangular tiling are not 2-colorable, and so we may have adjacent vertices with the same color, provided that the triangular Tangle has no adjacent links from the corresponding circles.

These colored vertices are also vertices of the boundary of a
\textit{polyform}, i.e., a subset of polygons belonging to the tiling.
We call this polyform, together with the vertex coloring of its boundary,
the \textit{dual polyform} of the Tangle.
By construction, this polyform is \textit{simply connected}, i.e., it
contains no holes.
The \textit{size} of a Tangle is the number of polygons comprising its
dual polyform.

Depending on the underlying tiling, polyforms are known by various names.
A square Tangle has a dual \textit{pseudo-polyomino} \cite{golomb} (or
\textit{polyplet} \cite{malkis}), a hexagonal Tangle has a dual
\textit{polyhex} \cite{HR}, and a triangular Tangle has a dual
\textit{pseudo-polyiamond} (or \textit{polyming} \cite{sicherman} or \textit{polyglass} \cite[A319324]{oeis}).
Note the \textit{pseudo-} prefixes in the first and third cases.
In contrast to their more well-known cousins, polyominoes and
polyiamonds, which only allow edge-to-edge connections between their
constituent polygons, the pseudo- variations also allow
vertex-to-vertex connections.
This is not an issue for polyhexes, as any two hexagons in the
hexagonal tiling that share a vertex must also share an edge.

In the square and triangular cases, the vertices where these
vertex-to-vertex connections occur are \textit{cut vertices} in the
underlying graph, i.e., their removal would create multiple connected
components.

\begin{lemma}\label{no white cut vertices}
  The dual polyform of a Tangle has no white cut vertices.
\end{lemma}

\begin{proof}
  A white cut vertex would correspond to a circle lying on the
  exterior of multiple Tangles.
  See, e.g., \cref{white cut vertex example}
\end{proof}

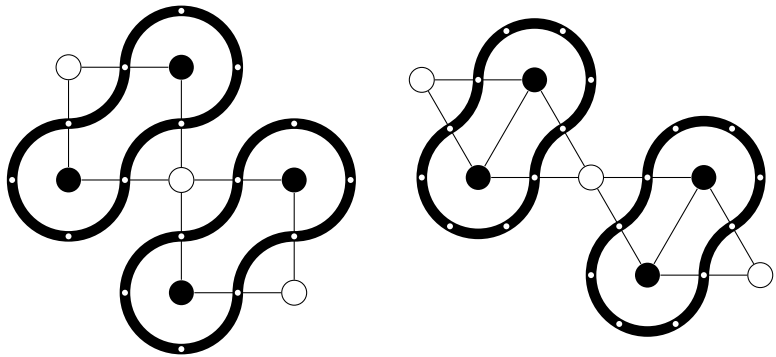
\begin{figure}
  \begin{tikzpicture}[scale=0.75]
    \node[shape=circle, fill] (a) at (-1,-1) {};
    \node[shape=circle, draw] (b) at (1,-1) {};
    \node[shape=circle, fill] (c) at (1,1) {};
    \node[shape=circle, draw] (d) at (-1,1) {};
    \node[shape=circle, fill] (e) at (-1,3) {};
    \node[shape=circle, fill] (f) at (-3,1) {};
    \node[shape=circle, draw] (g) at (-3,3) {};
    \draw (d) -- (a) -- (b) -- (c) -- (d) -- (e) -- (g) -- (f) -- (d);
    \draw[line width=4pt] (-1,0) arc (-90:0:1) arc (180:0:1) arc (0:-90:1)
      arc (90:180:1) arc (360:90:1);
    \draw[line width=4pt] (-3,2) arc (-90:0:1) arc (180:0:1) arc (0:-90:1)
      arc (90:180:1) arc (360:90:1);
    \fill[white] (-4,1) circle (1.5pt);
    \fill[white] (-3,0) circle (1.5pt);
    \fill[white] (-3,2) circle (1.5pt);
    \fill[white] (-2,-1) circle (1.5pt);
    \fill[white] (-2,1) circle (1.5pt);
    \fill[white] (-2,3) circle (1.5pt);
    \fill[white] (-1,-2) circle (1.5pt);
    \fill[white] (-1,0) circle (1.5pt);
    \fill[white] (-1,2) circle (1.5pt);
    \fill[white] (-1,4) circle (1.5pt);
    \fill[white] (0,-1) circle (1.5pt);
    \fill[white] (0,1) circle (1.5pt);
    \fill[white] (0,3) circle (1.5pt);
    \fill[white] (1,0) circle (1.5pt);
    \fill[white] (1,2) circle (1.5pt);
    \fill[white] (2,1) circle (1.5pt);
  \end{tikzpicture}\quad\quad
  \begin{tikzpicture}[scale=0.75]
    \node[shape=circle, fill] (a) at (0, 0) {};
    \node[shape=circle, draw] (b) at (2, 0) {};
    \node[shape=circle, fill] (c) at (4, 0) {};
    \node[shape=circle, draw] (d) at ({-2*cos(60)}, {2*sin(60)}) {};
    \node[shape=circle, fill] (e) at ({2*cos(60)}, {2*sin(60)}) {};
    \node[shape=circle, fill] (f) at ({2+2*cos(60)}, {-2*sin(60)}) {};
    \node[shape=circle, draw] (g) at ({4+2*cos(60)}, {-2*sin(60)}) {};
    \draw[line width=4pt] (1, 0) arc (360:120:1) arc (-60:0:1) arc (180:-60:1) arc (120:180:1);
    \draw[line width=4pt] (4, {-2*sin(60)}) arc (360:120:1) arc (-60:0:1) arc (180:-60:1) arc (120:180:1);
    \draw (a) -- (b) -- (c) -- (g) -- (f) -- (b) -- (e) -- (d) -- (a);
    \draw (a) -- (e);
    \draw (f) -- (c);
    \fill[white] (-1, 0) circle (1.5pt);
    \fill[white] (1, 0) circle (1.5pt);
    \fill[white] (3, 0) circle (1.5pt);
    \fill[white] (5, 0) circle (1.5pt);
    \fill[white] ({cos(120)}, {sin(120)}) circle (1.5pt);
    \fill[white] ({2+cos(120)}, {sin(120)}) circle (1.5pt);
    \fill[white] ({4+cos(120)}, {sin(120)}) circle (1.5pt);
    \fill[white] ({4+cos(60)}, {sin(120)}) circle (1.5pt);
    \fill[white] (0, {2*sin(60)}) circle (1.5pt);
    \fill[white] (2, {2*sin(60)}) circle (1.5pt);
    \fill[white] ({cos(60)}, {3*sin(60)}) circle (1.5pt);
    \fill[white] ({1+cos(60)}, {3*sin(60)}) circle (1.5pt);
    \fill[white] (-0.5, {-sin(60)}) circle (1.5pt);
    \fill[white] (0.5, {-sin(60)}) circle (1.5pt);
    \fill[white] (2.5, {-sin(60)}) circle (1.5pt);
    \fill[white] (4.5, {-sin(60)}) circle (1.5pt);
    \fill[white] (2, {-2*sin(60)}) circle (1.5pt);
    \fill[white] (4, {-2*sin(60)}) circle (1.5pt);
    \fill[white] (2.5, {-3*sin(60)}) circle (1.5pt);
    \fill[white] (3.5, {-3*sin(60)}) circle (1.5pt);
    \fill[white] (0, -3) circle (4pt); 
  \end{tikzpicture}
  \caption{White cut vertices imply multiple Tangles}
  \label{white cut vertex example}
\end{figure}

In the next three sections, we will explore what happens when we add
polygons to the dual polyform of a regular Tangle.  But first, we close this section with a generalization of a result proven for square Tangles in \cite{fleron}.

\begin{theorem}[Gauss-Bonnet theorem for Tangles]\label{gauss-bonnet}
  A Tangle with $\frac{1}{n}$-circle links has exactly $n$ more convex links than concave links.
\end{theorem}

\begin{proof}
  Let $T$ be such a Tangle with $j$ convex links and $k$ concave links.  Suppose $T$ has radius $r$, and so each link has arc length $\frac{2\pi r}{n}$.  Let $\vb x:\left[0, \frac{2\pi r(j + k)}{n}\right]\to T$ be an arc length parametrization of $T$.  Then, if $\kappa(s)$ is the signed curvature of $T$ at $\vb x(s)$, we have, provided that $\vb x(s)$ is not one of the finitely many intersection points between a convex link and a concave link,
  \begin{equation*}
    \kappa(s) = \begin{cases}
      \frac{1}{r} &\text{if $\vb x(s)$ lies in a convex link} \\
      -\frac{1}{r} &\text{if $\vb x(s)$ lies in a concave link}
    \end{cases}
  \end{equation*}

  Therefore,
  \begin{equation*}
    \oint_T\kappa(s)\,ds = \frac{1}{r}\cdot\frac{2\pi rj}{n} - \frac{1}{r}\cdot\frac{2\pi rk}{n} = \frac{2\pi(j - k)}{n},
  \end{equation*}
  and so by the Gauss-Bonnet theorem \cite[\S 4.8]{struik}, $\frac{2\pi(j-k)}{n} = 2\pi$, or $j - k = n$.
\end{proof}

\section{Square Tangles}\label{square tangles}

In \cite{taylor,torrance2019enumeration}, another combinatorial object, the
\textit{dual graph}, was used to describe a square Tangle.
However, dual graphs and dual pseudo-polyominoes are equivalent.
Indeed, we may form the dual graph from the dual pseudo-polyomino by first
coloring the interior vertices to agree with the existing coloring on
the boundary (i.e., no adjacent vertices may have the same color),
drawing an edge between every pair of black vertices belonging to the
same square, and then removing all the original edges and white
vertices.  See \cref{dual graph}.

\begin{figure}
  \begin{tikzpicture}[scale=0.75]
    \node[shape=circle,fill] (a) at (0, 0) {};
    \node[shape=circle,draw] (b) at (2, 0) {};
    \node[shape=circle,fill] (c) at (4, 0) {};
    \node[shape=circle,draw] (d) at (0, 2) {};
    \node[shape=circle,draw] (f) at (4, 2) {};
    \node[shape=circle,fill] (g) at (0, 4) {};
    \node[shape=circle,draw] (h) at (2, 4) {};
    \node[shape=circle,fill] (i) at (4, 4) {};
    \draw (a) -- (b) -- (c) -- (f) -- (i) -- (h) -- (g) -- (d) -- (a);
    \draw (b) -- (h);
    \draw (d) -- (f);
    \draw[line width=4pt] (1, 0) arc (360:90:1) arc (-90:90:1) arc (270:0:1) arc (180:360:1) arc (180:-90:1) arc (90:270:1) arc (90:-180:1) arc (0:180:1);
    \fill[white] (0, -1) circle (1.5pt);
    \fill[white] (4, -1) circle (1.5pt);
    \fill[white] (-1, 0) circle (1.5pt);
    \fill[white] (1, 0) circle (1.5pt);
    \fill[white] (3, 0) circle (1.5pt);
    \fill[white] (5, 0) circle (1.5pt);
    \fill[white] (0, 1) circle (1.5pt);
    \fill[white] (4, 1) circle (1.5pt);
    \fill[white] (1, 2) circle (1.5pt);
    \fill[white] (3, 2) circle (1.5pt);
    \fill[white] (0, 3) circle (1.5pt);
    \fill[white] (4, 3) circle (1.5pt);
    \fill[white] (-1, 4) circle (1.5pt);
    \fill[white] (1, 4) circle (1.5pt);
    \fill[white] (3, 4) circle (1.5pt);
    \fill[white] (5, 4) circle (1.5pt);
    \fill[white] (0, 5) circle (1.5pt);
    \fill[white] (4, 5) circle (1.5pt);
  \end{tikzpicture}\quad\quad
  \begin{tikzpicture}[scale=0.75]
    \node[shape=circle,fill] (a) at (0, 0) {};
    \node[shape=circle,fill] (c) at (4, 0) {};
    \node[shape=circle,fill] (e) at (2, 2) {};
    \node[shape=circle,fill] (g) at (0, 4) {};
    \node[shape=circle,fill] (i) at (4, 4) {};
    \draw (a) -- (i);
    \draw (c) -- (g);
    \draw[line width=4pt] (1, 0) arc (360:90:1) arc (-90:90:1) arc (270:0:1) arc (180:360:1) arc (180:-90:1) arc (90:270:1) arc (90:-180:1) arc (0:180:1);
    \fill[white] (0, -1) circle (1.5pt);
    \fill[white] (4, -1) circle (1.5pt);
    \fill[white] (-1, 0) circle (1.5pt);
    \fill[white] (1, 0) circle (1.5pt);
    \fill[white] (3, 0) circle (1.5pt);
    \fill[white] (5, 0) circle (1.5pt);
    \fill[white] (0, 1) circle (1.5pt);
    \fill[white] (4, 1) circle (1.5pt);
    \fill[white] (1, 2) circle (1.5pt);
    \fill[white] (3, 2) circle (1.5pt);
    \fill[white] (0, 3) circle (1.5pt);
    \fill[white] (4, 3) circle (1.5pt);
    \fill[white] (-1, 4) circle (1.5pt);
    \fill[white] (1, 4) circle (1.5pt);
    \fill[white] (3, 4) circle (1.5pt);
    \fill[white] (5, 4) circle (1.5pt);
    \fill[white] (0, 5) circle (1.5pt);
    \fill[white] (4, 5) circle (1.5pt);
  \end{tikzpicture}
  \caption{Relationship between the dual pseudo-polyomino and the dual graph of a square Tangle}
  \label{dual graph}
\end{figure}
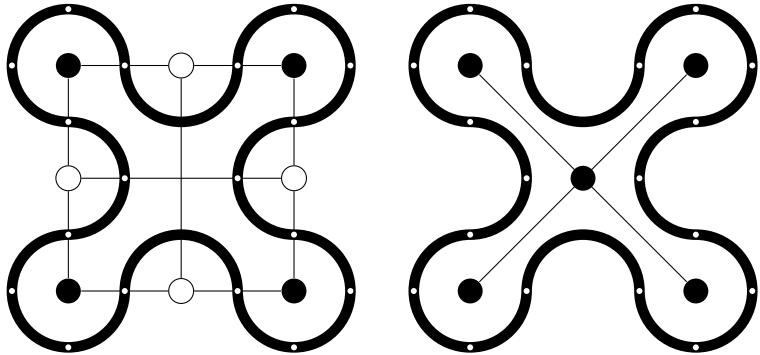

There are two fundamental operations on square Tangles that correspond to adding a single square to the dual pseudo-polyomino.
By \cref{no white cut vertices}, the new square must share at least
one black vertex with the existing dual pseudo-polyomino.
If the new square shares a single black vertex with the existing
dual pseudo-polyomino, then we add a 3-bulb.
Note that this corresponds to adding a leaf edge to the dual graph
as in \cite[Lemma 2.1]{torrance2019enumeration}.
If the new square shares two black vertices with the existing dual
pseudo-polyomino, then we remove an inverted 3-bulb.
This corresponds to adding an edge that completes a square in the
dual graph as in \cite[Lemma 2.2]{torrance2019enumeration}.
These two operations are referred to as \textit{shear insertion}
and \textit{shear reduction}, respectively, in \cite{fleron}.
See \cref{square tangle operations}.  In this figure, gray indicates the links that are being removed and black the links that are being added.

\begin{figure}
  \begin{tikzpicture}[scale=0.75]
    \node[shape=circle, fill] (a) at (-1,-1) {};
    \node[shape=circle, draw] (b) at (1,-1) {};
    \node[shape=circle, fill] (c) at (1,1) {};
    \node[shape=circle, draw] (d) at (-1,1) {};
    \draw (a) -- (b) -- (c) -- (d) -- (a);
    \draw[line width=4pt, color=lightgray] (-1,0) arc (90:0:1);
    \draw[line width=4pt] (0, -1) arc (180:90:1) arc (-90:180:1) arc (0:-90:1);
    \fill[white] (0,-1) circle (1.5pt);
    \fill[white] (1,0) circle (1.5pt);
    \fill[white] (2,1) circle (1.5pt);
    \fill[white] (1,2) circle (1.5pt);
    \fill[white] (0,1) circle (1.5pt);
    \fill[white] (-1,0) circle (1.5pt);

    \node[shape=circle, fill] (e) at (3,-1) {};
    \node[shape=circle, draw] (f) at (5,-1) {};
    \node[shape=circle, fill] (g) at (5,1) {};
    \node[shape=circle, draw] (h) at (3,1) {};
    \draw (e) -- (f) -- (g) -- (h) -- (e);
    \draw[line width=4pt, color=lightgray] (3, 0) arc (90:0:1) arc (-180:90:1)
    arc (270:180:1);
    \draw[line width=4pt] (3,0) arc (-90:0:1);
    \fill[white] (3,0) circle (1.5pt);
    \fill[white] (4,-1) circle (1.5pt);
    \fill[white] (5,-2) circle (1.5pt);
    \fill[white] (6,-1) circle (1.5pt);
    \fill[white] (5,0) circle (1.5pt);
    \fill[white] (4,1) circle (1.5pt);
  \end{tikzpicture}
  \caption{Shear insertion (left) and shear reduction (right)}
  \label{square tangle operations}
\end{figure}
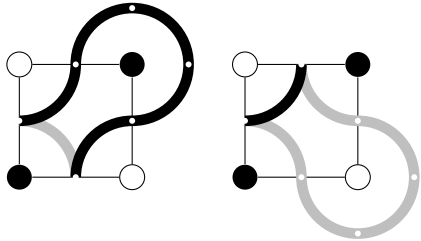

A third operation on square Tangles mentioned in the literature
(called \textit{reflection} in \cite{fleron} and
\textit{$\Omega$-rotation} in \cite{chan}) is actually a composition
of these two fundamental operations.
See \cref{square tangle reflection}.

\begin{figure}
  \begin{tikzpicture}[scale=0.75]
    \node[shape=circle, fill] (a) at (-1,-1) {};
    \node[shape=circle, draw] (b) at (1,-1) {};
    \node[shape=circle, fill] (c) at (1,1) {};
    \node[shape=circle, draw] (d) at (-1,1) {};
    \node[shape=circle, fill] (e) at (3,-1) {};
    \node[shape=circle, draw] (f) at (3,1) {};
    \draw (c) -- (b) -- (a) -- (d) -- (c) -- (f) -- (e) -- (b);
    \draw[line width=4pt, color=lightgray] (-1,0) arc (90:0:1) arc (180:360:1)
      arc (180:90:1);
    \draw[line width=4pt] (-1,0) arc (-90:0:1) arc (180:0:1) arc (180:270:1);
    \fill[white] (-1,0) circle (1.5pt);
    \fill[white] (0,1) circle (1.5pt);
    \fill[white] (0,-1) circle (1.5pt);
    \fill[white] (1,2) circle (1.5pt);
    \fill[white] (1,-2) circle (1.5pt);
    \fill[white] (2,1) circle (1.5pt);
    \fill[white] (2,-1) circle (1.5pt);
    \fill[white] (3,0) circle (1.5pt);
  \end{tikzpicture}

  \caption{Reflection or $\Omega$-rotation}
  \label{square tangle reflection}
\end{figure}
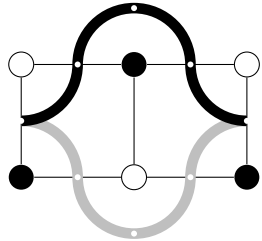

\begin{definition}
  A square Tangle is \textit{constructible} if it may be obtained from a circle by a sequence of shear insertions and/or reductions.
\end{definition}

The following result is essentially a special case of \cite[Lemma 2]{DDEFF}, which was proven for polyforms for which the constituent polygons share an edge.  Of course, in our setting, we also allow them to share only a vertex.

\begin{theorem}\label{square tangle constructible}
  Every square Tangle is constructible.
\end{theorem}

\begin{proof}
  We proceed by induction on the size $m$ of the square Tangle.

  For the base case $m=0$, the dual pseudo-polyomino contains a single black vertex and no squares.  In other words, the Tangle is a circle, and the result is immediate.

  Now assume $m>0$.  Form a graph by assigning every square of the dual pseudo-polyomino a vertex and adding edges for every pair of squares that share either a black boundary vertex or an interior vertex.  By \cref{no white cut vertices}, this graph is connected, and so it has a spanning tree \cite[Theorem 4.10]{CZ}.  Every tree has at least one leaf vertex (or in the case $m=1$, a single isolated vertex).  Choose one of these and remove the corresponding square from the dual pseudo-polyomino.  The resulting square Tangle is constructible by induction, and so by adding back this square, we see that the original square Tangle is also constructible.
\end{proof}

This fact allows us to obtain new proofs of two well-known results.  The first of these is originally due to Wolfenden \cite{fleron}.
The use of induction similar to ours to give an alternate proof was suggested in \cite{FH}.

\begin{corollary}\label{square tangle length}
  The length of a square Tangle is a multiple of 4.
\end{corollary}

\begin{proof}
  Since a circle has length 4, a shear insertion results in a net increase of 4 links, and a shear reduction a net decrease of 4 links, the result follows immediately.
\end{proof}

For this reason, we define the \textit{class} of a square Tangle to be the number $c$ for which the length is $4c$.

The second result was originally proven in \cite{torrance2019enumeration}
using the dual graph and Euler's polyhedral formula.

\begin{corollary}\label{square tangle area}
  The area enclosed by a square Tangle of size $m$ and radius $r$ is
  $(4m+\pi)r^2$.
\end{corollary}

\begin{proof}
  The initial circle has area $\pi r^2$, and so it remains to determine the area added by each shear insertion or reduction.   In both cases, we are adding to the region enclosed by the square Tangle the area enclosed by a 3-bulb.  It can be seen by breaking this region into pieces and rearranging that its area is exactly that of a square with side length $2r$.   See \cref{3-bulb area}.
\end{proof}

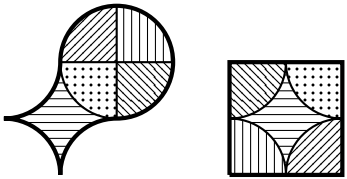
\begin{figure}
  \begin{tikzpicture}[scale=0.75]
    \fill[pattern=horizontal lines] (1,0) arc (270:180:1) arc
        (0:-90:1) arc (90:0:1) arc (180:90:1);
    \fill[pattern=north west lines] (1,0) arc (-90:0:1) -- (1,1);
    \fill[pattern=vertical lines] (2,1) arc (0:90:1) -- (1,1);
    \fill[pattern=north east lines] (0,1) arc (180:90:1) -- (1,1);
    \fill[pattern=dots] (1, 0) arc (270:180:1) -- (1, 1);
    \draw[thick] (0,1) -- (2,1);
    \draw[thick] (1, 0) -- (1, 2);
    \draw[thick] (1, 0) arc (270:180:1);
    \draw[ultra thick] (1,0) arc (-90:180:1) arc (0:-90:1) arc (90:0:1) arc
        (180:90:1);

    \fill[pattern=horizontal lines] (5,0) arc (270:180:1) arc
        (0:-90:1) arc (90:0:1) arc (180:90:1);
    \fill[pattern=north west lines] (3,0) arc (-90:0:1) -- (3,1);
    \fill[pattern=vertical lines] (4,-1) arc (0:90:1) -- (3,-1);
    \fill[pattern=north east lines] (4,-1) arc (180:90:1) -- (5,-1);
    \fill[pattern=dots] (5, 0) arc (270:180:1) -- (5, 1);
    \draw[thick] (4,1) arc (0:-90:1) arc (90:0:1) arc (180:90:1) arc (270:180:1);
    \draw[ultra thick] (3,-1) -- (5,-1) -- (5,1) -- (3,1) -- (3,-1);
  \end{tikzpicture}
  \caption{Area enclosed by a 3-bulb}\label{3-bulb area}
\end{figure}

\section{Hexagonal Tangles}\label{hexagonal tangles}

As mentioned in \cref{dual polyforms}, hexagons in a hexagonal tiling of a plane always share an edge, and so a new hexagon that has been added to the dual polyhex of a hexagonal Tangle will always share at least black vertex with the existing hexagons.  This results in three distinct fundamental operations.

If the new hexagon shares a single black vertex with the existing hexagons, then we add two 2-bulbs that share a concave link.  If it shares two black vertices with the existing hexagons, then we flip an inverted 2-bulb into a 2-bulb.  If it shares all three black vertices with the existing hexagons, then we remove two inverted 2-bulbs that share a convex link.  We denote these operations as \textit{shear insertion}, \textit{reflection}, and \textit{shear reduction}, respectively, in analogy with the operations on square Tangles.  See \cref{hexagonal tangle operations}.  As before, gray indicates links that are being removed and black links that are being added.

\begin{figure}
  \begin{tikzpicture}[scale=0.75]
    \node[shape=circle, fill] (a) at (0:2) {};
    \node[shape=circle, draw] (b) at (60:2) {};
    \node[shape=circle, fill] (c) at (120:2) {};
    \node[shape=circle, draw] (d) at (180:2) {};
    \node[shape=circle, fill] (e) at (240:2) {};
    \node[shape=circle, draw] (f) at (300:2) {};
    \draw (a) -- (b) -- (c) -- (d) -- (e) -- (f) -- (a);
    \draw[line width=4pt,color=lightgray] (0, {-sqrt(3)}) arc (0:120:1);
    \draw[line width=4pt] (0, {-sqrt(3)}) arc (180:60:1) arc (-120:120:1) arc (300:180:1) arc (0:240:1) arc (60:-60:1);
    \fill[white] (0:3) circle (1.5pt);
    \fill[white] (30:{sqrt(3)}) circle (1.5pt);
    \fill[white] (90:{sqrt(3)}) circle (1.5pt);
    \fill[white] (120:3) circle (1.5pt);
    \fill[white] (150:{sqrt(3)}) circle (1.5pt);
    \fill[white] (210:{sqrt(3)}) circle (1.5pt);
    \fill[white] (270:{sqrt(3)}) circle (1.5pt);
    \fill[white] (330:{sqrt(3)}) circle (1.5pt);
    \fill[white] (0, {-sqrt(3) - 1}) circle (2pt); 
  \end{tikzpicture}
  \begin{tikzpicture}[scale=0.75]
    \node[shape=circle, fill] (a) at (0:2) {};
    \node[shape=circle, draw] (b) at (60:2) {};
    \node[shape=circle, fill] (c) at (120:2) {};
    \node[shape=circle, draw] (d) at (180:2) {};
    \node[shape=circle, fill] (e) at (240:2) {};
    \node[shape=circle, draw] (f) at (300:2) {};
    \draw (a) -- (b) -- (c) -- (d) -- (e) -- (f) -- (a);
    \draw[line width=4pt,color=lightgray] (210:{sqrt(3)}) arc (120:0:1) arc (-180:60:1) arc (240:120:1);
    \draw[line width=4pt] (30:{sqrt(3)}) arc (300:180:1) arc (0:240:1) arc (60:-60:1);
    \fill[white] (30:{sqrt(3)}) circle (1.5pt);
    \fill[white] (90:{sqrt(3)}) circle (1.5pt);
    \fill[white] (120:3) circle (1.5pt);
    \fill[white] (150:{sqrt(3)}) circle (1.5pt);
    \fill[white] (210:{sqrt(3)}) circle (1.5pt);
    \fill[white] (270:{sqrt(3)}) circle (1.5pt);
    \fill[white] (300:3) circle (1.5pt);
    \fill[white] (330:{sqrt(3)}) circle (1.5pt);
  \end{tikzpicture}
  \begin{tikzpicture}[scale=0.75]
    \node[shape=circle, fill] (a) at (0:2) {};
    \node[shape=circle, draw] (b) at (60:2) {};
    \node[shape=circle, fill] (c) at (120:2) {};
    \node[shape=circle, draw] (d) at (180:2) {};
    \node[shape=circle, fill] (e) at (240:2) {};
    \node[shape=circle, draw] (f) at (300:2) {};
    \draw (a) -- (b) -- (c) -- (d) -- (e) -- (f) -- (a);
    \draw[line width=4pt,color=lightgray] (210:{sqrt(3)}) arc (120:0:1) arc (-180:60:1) arc (240:120:1) arc (-60:180:1) arc (0:-120:1);
    \draw[line width=4pt] (150:{sqrt(3)}) arc (60:-60:1);
    \fill[white] (30:{sqrt(3)}) circle (1.5pt);
    \fill[white] (60:3) circle (1.5pt);
    \fill[white] (90:{sqrt(3)}) circle (1.5pt);
    \fill[white] (150:{sqrt(3)}) circle (1.5pt);
    \fill[white] (210:{sqrt(3)}) circle (1.5pt);
    \fill[white] (270:{sqrt(3)}) circle (1.5pt);
    \fill[white] (300:3) circle (1.5pt);
    \fill[white] (330:{sqrt(3)}) circle (1.5pt);
  \end{tikzpicture}
  \caption{Shear insertion (left), reflection (center), and shear reduction (right)}
  \label{hexagonal tangle operations}
\end{figure}
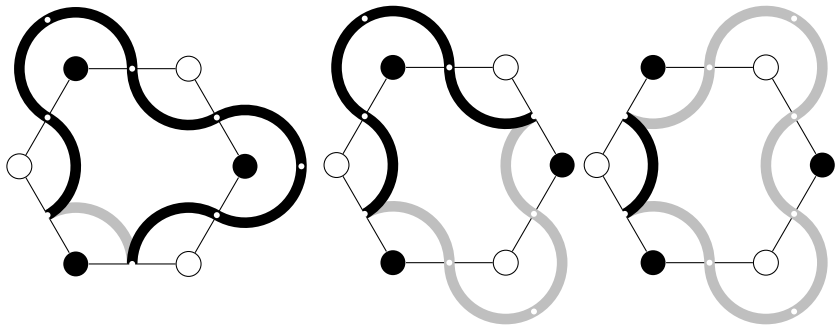

\begin{definition}
  A hexagonal Tangle is \textit{constructible} if it may be obtained from a circle by a sequence of shear insertions, reflections, and/or shear reductions.
\end{definition}

The following result is immediate using the same proof as that of \cref{square tangle constructible}, but replacing each instance of the word ``square'' with ``hexagon'' and ``pseudo-polyomino'' with ``polyhex''.

\begin{theorem}
  Every hexagonal Tangle is constructible.
\end{theorem}

We next obtain the following analogues to Corollaries \ref{square tangle length} and \ref{square tangle area}.

\begin{corollary}\label{hexagonal tangle length}
  The length of a hexagonal Tangle is congruent to 3 modulo 6.
\end{corollary}

\begin{proof}
  The initial circle has length 3, each shear insertion increases the length by 6, each reflection results in no change in length, and each shear reduction decreases the length by 6.
\end{proof}

\begin{corollary}\label{hexagonal tangle area}
  The area enclosed by a hexagonal Tangle of size $m$ and radius $r$ is $(6\sqrt 3 m + \pi)r^2$.
\end{corollary}

\begin{proof}
  The area of the initial circle is $\pi r^2$.  For each fundamental operation, we add the area enclosed by two 2-bulbs.  As can be seen in \cref{two 2-bulb area}, this area is equal to that enclosed by a regular hexagon with side length $2r$, and so has area $6\sqrt 3 r^2$.
\end{proof}

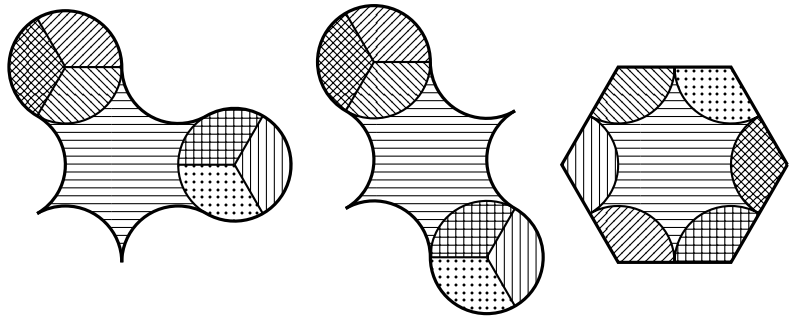
\begin{figure}
  \begin{tikzpicture}[scale=0.75]
    \fill[pattern=horizontal lines] (30:{sqrt(3)}) arc (300:180:1) arc (360:240:1) arc (60:-60:1) arc (120:0:1) arc (180:60:1) arc (240:120:1);
    \fill[pattern=north west lines] (120:2) -- (90:{sqrt(3)}) arc (360:240:1);
    \fill[pattern=crosshatch] (120:2) -- (150:{sqrt(3)}) arc (240:120:1);
    \fill[pattern=north east lines] (120:2) -- (120:3) arc (120:0:1);
    \fill[pattern=grid] (0:2) -- (0:1) arc (180:60:1);
    \fill[pattern=vertical lines] (0:2) -- (2.5, {sqrt(3)/2}) arc (60:-60:1);
    \fill[pattern=dots] (0:2) -- (2.5, {-sqrt(3)/2}) arc (300:180:1);
    \draw[very thick] (0, {-sqrt(3)}) arc (0:120:1);
    \draw[very thick] (0, {-sqrt(3)}) arc (180:60:1) arc (-120:120:1) arc (300:180:1) arc (0:240:1) arc (60:-60:1);
    \draw[thick] (-30:{sqrt(3)}) arc (240:120:1);
    \draw[thick] (90:{sqrt(3)}) arc (360:240:1);
    \draw[thick] (120:2) -- (90:{sqrt(3)});
    \draw[thick] (120:2) -- (150:{sqrt(3)});
    \draw[thick] (120:2) -- (120:3);
    \draw[thick] (0:2) -- (0:1);
    \draw[thick] (0:2) -- ({2 + 1/2}, {sqrt(3)/2});
    \draw[thick] (0:2) -- ({2 + 1/2}, {-sqrt(3)/2});
    \fill[white] (0, {-1 - sqrt(3)}) circle (1.5pt); 
  \end{tikzpicture}
  \begin{tikzpicture}[scale=0.75]
    \fill[pattern=horizontal lines] (30:{sqrt(3)}) arc (300:180:1) arc (360:240:1) arc (60:-60:1) arc (120:0:1) arc (180:60:1) arc (240:120:1);
    \fill[pattern=north west lines] (120:2) -- (90:{sqrt(3)}) arc (360:240:1);
    \fill[pattern=crosshatch] (120:2) -- (150:{sqrt(3)}) arc (240:120:1);
    \fill[pattern=north east lines] (120:2) -- (120:3) arc (120:0:1);
    \fill[pattern=grid] (300:2) -- (270:{sqrt(3)}) arc (180:60:1);
    \fill[pattern=vertical lines] (300:2) -- (1.5, {-sqrt(3)/2}) arc (60:-60:1);
    \fill[pattern=dots] (300:2) -- (1.5, {-3*sqrt(3)/2}) arc (300:180:1);
    \draw[very thick] (210:{sqrt(3)}) arc (120:0:1) arc (-180:60:1) arc (240:120:1);
    \draw[very thick] (30:{sqrt(3)}) arc (300:180:1) arc (0:240:1) arc (60:-60:1);
    \draw[thick] (-30:{sqrt(3)}) arc (60:180:1);
    \draw[thick] (90:{sqrt(3)}) arc (360:240:1);
    \draw[thick] (120:2) -- (90:{sqrt(3)});
    \draw[thick] (120:2) -- (150:{sqrt(3)});
    \draw[thick] (120:2) -- (120:3);
    \draw[thick] (300:2) -- (270:{sqrt(3)});
    \draw[thick] (300:2) -- (1.5, {-sqrt(3)/2});
    \draw[thick] (300:2) -- (1.5, {-3*sqrt(3)/2});
  \end{tikzpicture}
  \begin{tikzpicture}[scale=0.75]
    \fill[pattern=horizontal lines] (30:{sqrt(3)}) arc (300:180:1) arc (360:240:1) arc (60:-60:1) arc (120:0:1) arc (180:60:1) arc (240:120:1);
    \fill[pattern=dots] (30:{sqrt(3)}) arc (300:180:1) -- (60:2);
    \fill[pattern=north west lines] (90:{sqrt(3)}) arc (360:240:1) -- (120:2);
    \fill[pattern=vertical lines] (150:{sqrt(3)}) arc (60:-60:1) -- (180:2);
    \fill[pattern=north east lines] (210:{sqrt(3)}) arc (120:0:1) -- (240:2);
    \fill[pattern=grid] (270:{sqrt(3)}) arc (180:60:1) -- (300:2);
    \fill[pattern=crosshatch] (330:{sqrt(3)}) arc (240:120:1) -- (0:2);
    \draw[very thick] (0:2) -- (60:2) -- (120:2) -- (180:2) -- (240:2) -- (300:2) -- (0:2);
    \draw[thick] (30:{sqrt(3)}) arc (300:180:1) arc (360:240:1) arc (60:-60:1) arc (120:0:1) arc (180:60:1) arc (240:120:1);
    \fill[white] (0, {-1 - sqrt(3)}) circle (1.5pt); 
  \end{tikzpicture}

  \caption{Area enclosed by two 2-bulbs}
  \label{two 2-bulb area}
\end{figure}

\section{Triangular Tangles}\label{triangular tangles}

We now move to triangular Tangles.  They are trickier than the other two families, as vertices of boundaries of pseudo-polyiamonds are in general not 2-colorable.  In particular, we cannot simply add a single triangle to the dual pseudo-polyiamond of a triangular Tangle and preserve the chromatic number of its boundary.  Instead, we will need to add two triangles simultaneously.  We will see that it is sufficient to add two triangles that share at least one vertex (a \textit{pseudo-diamond}).

There are three distinct pseudo-diamonds: the \textit{diamond}, in which the triangles share an edge, and two pseudo-diamonds for which the triangles share only a vertex.  We will call the the pseudo-diamond in which the edges of the triangles are separated by $120^\circ$ the \textit{bowtie} and the one in which an edge from each triangle lies on a common line the \textit{vampire teeth}.

The addition of these three pseudo-diamonds to the dual pseudo-polyiamond of a triangular Tangle corresponds to five fundamental operations.

If we add a diamond that shares a single black vertex (shared by both triangles) with the existing pseudo-polyiamond, then we add a 4-bulb.  If we add a diamond that shares two black vertices (one from each triangle, not incident to the shared edge), then we remove an inverted 4-bulb.  We denote these operations \textit{4-bulb shear insertion} and \textit{4-bulb shear reduction}, respectively.

If we add vampire teeth that share two black vertices (one from each triangle, not on the common line) with the existing pseudo-polyiamond, then we add a 3-bulb.  If we add vampire teeth that share two black vertices (one from each triangle, on the common line), then we remove an inverted 3-bulb.  These operations will be denoted as \textit{3-bulb shear insertion} and \textit{3-bulb shear reduction}, respectively.

If we add a bowtie that shares two black vertices (one from each triangle, not the shared vertex) with the existing pseudo-polyiamond, then we switch an inverted 2-bulb to a 2-bulb.  As with the corresponding operations for square and hexagonal Tangles, this is a \textit{reflection}.

See \cref{triangular tangle operations} for an illustration of these five fundamental operations.  Note some of them result in one of the existing vertices changing color, i.e., a circle switching from exterior to interior.  These are denoted using gray.

\begin{figure}
  \begin{tikzpicture}[scale=0.75]
    \node[shape=circle, fill] (a) at (0:0) {};
    \node[shape=circle, fill] (b) at (0:2) {};
    \node[shape=circle, draw] (c) at (60:2) {};
    \node[shape=circle, draw] (d) at (-60:2) {};
    \draw (a) -- (b) -- (c) -- (a) -- (d) --(b);
    \draw[line width=4pt,lightgray] (-60:1) arc (-60:60:1);
    \draw[line width=4pt] (-60:1) arc (120:60:1) arc (-120:120:1) arc (300:240:1);
    \fill[white] (-60:1) circle (1.5pt);
    \fill[white] (-30:{sqrt(3)}) circle (1.5pt);
    \fill[white] ({-asin(1/2*sqrt(3/7))}:{sqrt(7)}) circle (1.5pt);
    \fill[white] (0:3) circle (1.5pt);
    \fill[white] ({asin(1/2*sqrt(3/7))}:{sqrt(7)}) circle (1.5pt);
    \fill[white] (30:{sqrt(3)}) circle (1.5pt);
    \fill[white] (60:1) circle (1.5pt);
    \fill[white] (0:1) circle (1.5pt);
  \end{tikzpicture}
  \begin{tikzpicture}[scale=0.75]
    \node[shape=circle, draw] (a) at (0:0) {};
    \node[shape=circle, draw] (b) at (0:2) {};
    \node[shape=circle, fill] (c) at (60:2) {};
    \node[shape=circle, fill] (d) at (-60:2) {};
    \draw (a) -- (b) -- (c) -- (a) -- (d) --(b);
    \draw[line width=4pt,lightgray] (-60:1) arc (120:60:1) arc (-120:120:1) arc (300:240:1);
    \draw[line width=4pt] (-60:1) arc (-60:60:1);
    \fill[white] (-60:1) circle (1.5pt);
    \fill[white] (-30:{sqrt(3)}) circle (1.5pt);
    \fill[white] ({-asin(1/2*sqrt(3/7))}:{sqrt(7)}) circle (1.5pt);
    \fill[white] (0:3) circle (1.5pt);
    \fill[white] ({asin(1/2*sqrt(3/7))}:{sqrt(7)}) circle (1.5pt);
    \fill[white] (30:{sqrt(3)}) circle (1.5pt);
    \fill[white] (60:1) circle (1.5pt);
    \fill[white] (0:1) circle (1.5pt);
  \end{tikzpicture}

  \begin{tikzpicture}[scale=0.75]
    \node[shape=circle,draw] (a) at (180:2) {};
    \node[shape=circle,fill=lightgray,draw] (b) at (0:0) {};
    \node[shape=circle,draw] (c) at (0:2) {};
    \node[shape=circle,fill] (d) at (300:2) {};
    \node[shape=circle,fill] (e) at (240:2) {};
    \draw (a) -- (b) -- (c) -- (d) -- (b) -- (e) -- (a);
    \draw[line width=4pt,lightgray] (210:{sqrt(3)}) arc (120:60:1) arc (240:300:1) arc (120:60:1);
    \draw[line width=4pt] (210:{sqrt(3)}) arc (-60:0:1) arc (180:0:1) arc (180:240:1);
    \fill[white] (210:{sqrt(3)}) circle (1.5pt);
    \fill[white] (180:1) circle (1.5pt);
    \fill[white] (120:1) circle (1.5pt);
    \fill[white] (60:1) circle (1.5pt);
    \fill[white] (0:1) circle (1.5pt);
    \fill[white] (-30:{sqrt(3)}) circle (1.5pt);
    \fill[white] (240:1) circle (1.5pt);
    \fill[white] (300:1) circle (1.5pt);
  \end{tikzpicture}
  \begin{tikzpicture}[scale=0.75]
    \node[shape=circle,fill] (a) at (180:2) {};
    \node[shape=circle,fill=lightgray,draw] (b) at (0:0) {};
    \node[shape=circle,fill] (c) at (0:2) {};
    \node[shape=circle,draw] (d) at (300:2) {};
    \node[shape=circle,draw] (e) at (240:2) {};
    \draw (a) -- (b) -- (c) -- (d) -- (b) -- (e) -- (a);
    \draw[line width=4pt,lightgray] (210:{sqrt(3)}) arc (-60:0:1) arc (180:0:1) arc (180:240:1);
    \draw[line width=4pt] (210:{sqrt(3)}) arc (120:60:1) arc (240:300:1) arc (120:60:1);    \fill[white] (210:{sqrt(3)}) circle (1.5pt);
    \fill[white] (180:1) circle (1.5pt);
    \fill[white] (120:1) circle (1.5pt);
    \fill[white] (60:1) circle (1.5pt);
    \fill[white] (0:1) circle (1.5pt);
    \fill[white] (-30:{sqrt(3)}) circle (1.5pt);
    \fill[white] (240:1) circle (1.5pt);
    \fill[white] (300:1) circle (1.5pt);
  \end{tikzpicture}
  \begin{tikzpicture}[scale=0.75]
    \node[shape=circle,fill] (a) at (120:2) {};
    \node[shape=circle,draw] (b) at (60:2) {};
    \node[shape=circle,fill=lightgray,draw] (c) at (0:0) {};
    \node[shape=circle,draw] (d) at (300:2) {};
    \node[shape=circle,fill] (e) at (240:2) {};
    \draw (c) -- (a) -- (b) -- (c) -- (d) -- (e) -- (c);
    \draw[line width=4pt,lightgray] (90:{sqrt(3)}) arc (0:-60:1) arc (120:240:1) arc (60:0:1);
    \draw[line width=4pt] (90:{sqrt(3)}) arc (180:240:1) arc (60:-60:1) arc (120:180:1);
    \fill[white] (90:{sqrt(3)}) circle (1.5pt);
    \fill[white] (60:1) circle (1.5pt);
    \fill[white] (0:1) circle (1.5pt);
    \fill[white] (-60:1) circle (1.5pt);
    \fill[white] (270:{sqrt(3)}) circle (1.5pt);
    \fill[white] (120:1) circle (1.5pt);
    \fill[white] (180:1) circle (1.5pt);
    \fill[white] (240:1) circle (1.5pt);
  \end{tikzpicture}
  \caption{4-bulb shear insertion (top left), 4-bulb shear reduction (top right), 3-bulb shear insertion (bottom left), 3-bulb shear reduction (bottom center), reflection (bottom right)}
  \label{triangular tangle operations}
\end{figure}
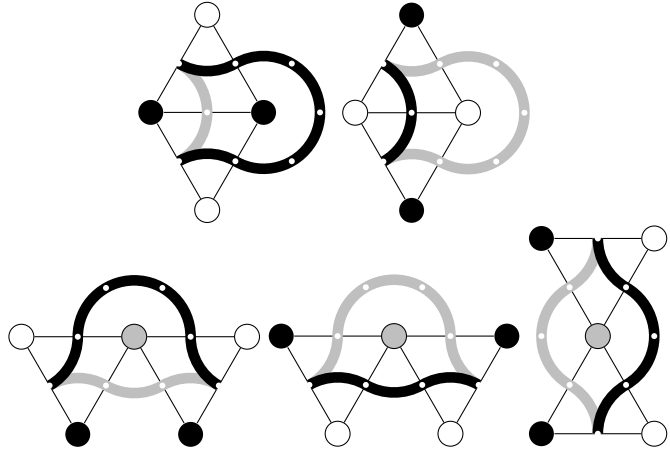

\begin{definition}
  A triangular Tangle is \textit{constructible} if it may be obtained from a circle by a sequence of 4-bulb shear insertions, 4-bulb shear reductions, 3-bulb shear insertions, 3-bulb shear reductions, and/or reflections.
\end{definition}

\begin{theorem}
  Every triangular Tangle is constructible.
\end{theorem}

\begin{proof}
  Suppose that there exists a non-constructible triangular Tangle of minimal size.  If it has a 4-bulb, then it may be obtained via 4-bulb shear insertion from a smaller Tangle, a contradiction.  Similarly, it cannot have a 3-bulb, as otherwise it may be obtained via 3-bulb shear insertion from a smaller Tangle.

  Suppose the Tangle contains a 2-bulb.  If this 2-bulb may be obtained via reflection from a smaller Tangle, then we again have a contradiction.

  However, it is possible for a triangular Tangle to contain a 2-bulb that cannot be obtained in this way if it contains at least two corresponding concave links from a kissing circle as seen in \cref{2-bulb no reflection}.  Suppose that the Tangle consists of only 2-bulbs of this type and/or 1-bulbs.  Then it must have at least as many concave links as convex links, contradicting \cref{gauss-bonnet}.
\end{proof}

\begin{figure}
  \begin{tikzpicture}[scale=0.75]
    \node[shape=circle,fill] (a) at (0, 0) {};
    \node[shape=circle,draw] (b) at (60:2) {};
    \node[shape=circle,draw] (d) at (180:2) {};
    \node[shape=circle,draw] (f) at (300:2) {};
    \draw (b) -- (120:2) -- (d) -- (240:2) -- (f);
    \draw (a) -- (b);
    \draw (a) -- (120:2);
    \draw (a) -- (d);
    \draw (a) -- (240:2);
    \draw (a) -- (f);
    \draw[line width=4pt] (150:{sqrt(3)}) arc (60:-60:1);
    \draw[line width=4pt] (270:{sqrt(3)}) arc (180:120:1) arc (-60:60:1) arc (240:180:1);
    \fill[white] (0, {-sqrt(3)}) circle (1.5pt);
    \fill[white] (-60:1) circle (1.5pt);
    \fill[white] (1, 0) circle (1.5pt);
    \fill[white] (60:1) circle (1.5pt);
    \fill[white] (0, {sqrt(3)}) circle (1.5pt);
    \fill[white] (150:{sqrt(3)}) circle (1.5pt);
    \fill[white] (-1, 0) circle (1.5pt);
    \fill[white] (210:{sqrt(3)}) circle (1.5pt);
  \end{tikzpicture}

  \caption{2-bulb that cannot be obtained via reflection}
 \label{2-bulb no reflection}
\end{figure}
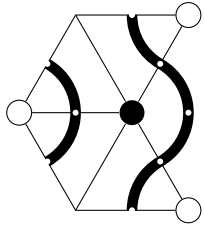

\begin{corollary}\label{triangular tangle length}
  The length of a triangular Tangle is even.
\end{corollary}

\begin{proof}
  The initial circle has length 6, each 4-bulb shear insertion increases the length by 4, each 3-bulb shear insertion increases the length by 2, each reflection preserves the length, each 3-bulb shear reduction decreases the length by 2, and each 4-bulb shear reduction decreases the length by 4.
\end{proof}

\begin{corollary}\label{triangular tangle area}
  The area enclosed by a triangular Tangle of size $m$ and radius $r$ is $(\sqrt 3m + \pi)r^2$.
\end{corollary}

\begin{proof}
  The area of the initial circle is $\pi r^2$.  For each fundamental operation, we add the area enclosed by a 4-bulb, 3-bulb, or 2-bulb and also increase the size by 2.  As seen in \cref{triangular tangle area cases}, each of these is the same as the area enclosed by two equilateral triangles with side length $2r$, and thus area $\sqrt 3 r^2$.
\end{proof}

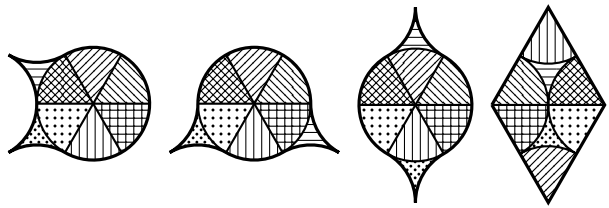
\begin{figure}
  \begin{tikzpicture}[scale=0.75]
    \fill[pattern=crosshatch] (0:0) -- (180:1) arc (180:120:1) -- (0:0);
    \fill[pattern=north east lines] (0:0) -- (120:1) arc (120:60:1) -- (0:0);
    \fill[pattern=north west lines] (0:0) -- (60:1) arc (60:0:1) -- (0:0);
    \fill[pattern=grid] (0:0) -- (300:1) arc (300:360:1) -- (0:0);
    \fill[pattern=vertical lines] (0:0) -- (240:1) arc (240:300:1) -- (0:0);
    \fill[pattern=dots] (0:0) -- (180:1) arc (180:240:1) -- (0:0);
    \fill[pattern=horizontal lines] (120:1) arc (300:240:1) arc (60:0:1) arc (180:120:1);
    \fill[pattern=crosshatch dots] (180:1) arc (0:-60:1) arc (120:60:1) arc (240:180:1);
    \draw[thick] (0:0) -- (0:1);
    \draw[thick] (0:0) -- (60:1);
    \draw[thick] (0:0) -- (120:1);
    \draw[thick] (0:0) -- (180:1);
    \draw[thick] (0:0) -- (240:1);
    \draw[thick] (0:0) -- (300:1);
    \draw[thick] (120:1) arc (120:240:1);
    \draw[very thick] (150:{sqrt(3)}) arc (60:-60:1) arc (120:60:1) arc (-120:120:1) arc (300:240:1);
    \fill[white] (270:{sqrt(3)}) circle (1.5pt); 
  \end{tikzpicture}
  \begin{tikzpicture}[scale=0.75]
    \fill[pattern=crosshatch] (0:0) -- (180:1) arc (180:120:1) -- (0:0);
    \fill[pattern=north east lines] (0:0) -- (120:1) arc (120:60:1) -- (0:0);
    \fill[pattern=north west lines] (0:0) -- (60:1) arc (60:0:1) -- (0:0);
    \fill[pattern=grid] (0:0) -- (300:1) arc (300:360:1) -- (0:0);
    \fill[pattern=vertical lines] (0:0) -- (240:1) arc (240:300:1) -- (0:0);
    \fill[pattern=dots] (0:0) -- (180:1) arc (180:240:1) -- (0:0);
    \fill[pattern=horizontal lines] (0:1) arc (180:240:1) arc (60:120:1) arc (-60:0:1);
    \fill[pattern=crosshatch dots] (180:1) arc (0:-60:1) arc (120:60:1) arc (240:180:1);
    \draw[thick] (0:0) -- (0:1);
    \draw[thick] (0:0) -- (60:1);
    \draw[thick] (0:0) -- (120:1);
    \draw[thick] (0:0) -- (180:1);
    \draw[thick] (0:0) -- (240:1);
    \draw[thick] (0:0) -- (300:1);
    \draw[thick] (0:1) arc (0:-60:1);
    \draw[thick] (180:1) arc (180:240:1);
    \draw[very thick] (180:1) arc (180:0:1) arc (180:240:1) arc (60:120:1) arc (300:240:1) arc (60:120:1) arc (-60:0:1);
    \fill[white] (270:{sqrt(3)}) circle (1.5pt); 
  \end{tikzpicture}
    \begin{tikzpicture}[scale=0.75]
    \fill[pattern=crosshatch] (0:0) -- (180:1) arc (180:120:1) -- (0:0);
    \fill[pattern=north east lines] (0:0) -- (120:1) arc (120:60:1) -- (0:0);
    \fill[pattern=north west lines] (0:0) -- (60:1) arc (60:0:1) -- (0:0);
    \fill[pattern=grid] (0:0) -- (300:1) arc (300:360:1) -- (0:0);
    \fill[pattern=vertical lines] (0:0) -- (240:1) arc (240:300:1) -- (0:0);
    \fill[pattern=dots] (0:0) -- (180:1) arc (180:240:1) -- (0:0);
    \fill[pattern=horizontal lines] (120:1) arc (-60:0:1) arc (180:240:1) arc (60:120:1);
    \fill[pattern=crosshatch dots] (240:1) arc (240:300:1) arc (120:180:1) arc (0:60:1);
    \draw[thick] (0:0) -- (0:1);
    \draw[thick] (0:0) -- (60:1);
    \draw[thick] (0:0) -- (120:1);
    \draw[thick] (0:0) -- (180:1);
    \draw[thick] (0:0) -- (240:1);
    \draw[thick] (0:0) -- (300:1);
    \draw[thick] (120:1) arc (120:60:1);
    \draw[thick] (240:1) arc (240:300:1);
    \draw[very thick] (60:1) arc (60:-60:1) arc (120:180:1) arc (0:60:1) arc (240:120:1) arc (-60:0:1) arc (180:240:1);
  \end{tikzpicture}
  \begin{tikzpicture}[scale=0.75]
    \fill[pattern=crosshatch] (1, 0) arc (180:120:1) -- (2, 0);
    \fill[pattern=north east lines] (-60:1) arc (120:60:1) -- (-60:2);
    \fill[pattern=north west lines] (1, 0) arc (0:60:1) -- (0, 0);
    \fill[pattern=grid] (1, 0) arc (360:300:1) -- (0, 0);
    \fill[pattern=vertical lines] (60:1) arc (240:300:1) -- (60:2);
    \fill[pattern=dots] (1, 0) arc (180:240:1) -- (2,0);
    \fill[pattern=horizontal lines] (1, 0) arc (180:120:1) arc (300:240:1) arc (60:0:1);
    \fill[pattern=crosshatch dots] (1, 0) arc (180:240:1) arc (60:120:1) arc (-60:0:1);
    \draw[very thick] (0:0) -- (60:2) -- (0:2) -- (-60:2) -- (0:0);
    \draw[thick] (0, 0) -- (2, 0);
    \draw[thick] (1, 0) arc (180:120:1) arc (300:240:1) arc (60:0:1);
    \draw[thick] (1, 0) arc (180:240:1) arc (60:120:1) arc (-60:0:1);
  \end{tikzpicture}
  \label{triangular tangle area cases}
  \caption{Area enclosed by a 4-bulb, 3-bulb, and 2-bulb}
\end{figure}

\bibliography{regular-tangles}{}
\bibliographystyle{plain}

\end{document}